\documentclass[preprint,10pt]{elsarticle}
\usepackage{amsmath,amsopn,amssymb,epsfig} 

\newtheorem{Theorem}{Theorem}[section]

\newtheorem{Proposition}{Proposition}[section]
\newtheorem{Remark}{Remark}[section]

\newtheorem{example}{Example}[section]

\newproof{proof}{Proof}
\newproof{pot}{Proof of Theorem \ref{thm2}}

\def\r2n{\mathbb{R}^{2n}}
\def\r2n2n{\mathbb{R}^{2n\times 2n}}

\newcommand{\bb}{\begin{bmatrix}}
\newcommand{\eb}{\end{bmatrix}}

\begin{document}

\date{}
\begin{frontmatter}
\title{
On the solution of the linear matrix equation $X=Af(X)B+C$
\tnoteref{t1}}
%
\author{Chun-Yueh Chiang\corref{cor1}\fnref{fn1}}
\ead{chiang@nfu.edu.tw}
\address{Center for General Education, National Formosa
University, Huwei 632, Taiwan.}


\cortext[cor1]{Corresponding author}
\fntext[fn1]{The author was supported  by the National Science Council of Taiwan under grant
NSC102-2115-M-150-002.}

\date{ }

\begin{abstract}
In this paper, we derive a formula to compute the solution of the linear matrix equation $X=Af(X)B+C$ via finding any solution of a specific Stein matrix equation $\mathcal{X}=\mathcal{A} \mathcal{X} \mathcal{B}+\mathcal{C}$, where the linear (or anti-linear) matrix operator $f$ is period-$n$. According to this formula, we should pay much attention to solve the Stein matrix equation from recently famous numerical methods. For instance, Smith-type iterations, Bartels-Stewart algorithm, and etc.. Moreover, this transformation is used to provide necessary and sufficient conditions of the solvable of the linear matrix equation.
On the other hand, it can be proven that the general solution of the linear matrix equation can be presented by the general solution of the Stein matrix equation. The necessary condition of the uniquely solvable of the linear matrix equation is developed. It is shown that several representations of this formula are coincident. Some examples are presented to illustrate and explain our results.
\end{abstract}

\end{frontmatter}

\section{Introduction}
The paper concerns with a general class of linear matrix equation of the form
\begin{align}\label{eq}
X=Af(X)B+C,
\end{align}
where $A,\,B \in\mathbb{C}^{m\times m}$ are known matrices, and $X\in \mathbb{C}^{m\times m}$ is an unknown matrix to be determined. The matrix operator $f:\mathbb{C}^{m\times m}\rightarrow \mathbb{C}^{m\times m}$ is satisfying the following properties.
\begin{itemize}
\item[1.] Linear operator or anti-linear operator: $f(AB+CD)=f(A)f(B)+f(C)f(D)$ or $f(B)f(A)+f(D)f(C)$ for all $A,B,C$ and $D\in\mathbb{C}^{m\times m}.$
\item[2.]Period-$n$: For all $A\in\mathbb{C}^{m\times m}$, $i=n$ is the least positive integer such that $f^{(i)}(A)=A$ .
\end{itemize}
In \cite{Zhou2011}, Zhou et al. consider the  solvability, existence of unique solution,
closed-form solution of the following three types of linear matrix equation.
\begin{itemize}
\item[1.]$X=AX^\top B+C$, $f$ is the transpose operator which is an anti-linear and period-$2$ operator.
\item[2.]$X=AX^H B+C$, $f$ is the Hermitian operator which is an anti-linear and period-$2$ operator.
\item[3.]$X=A\overline{X} B+C$, $f$ is the conjugate operator which is a linear and period-$2$ operator.
\end{itemize}
Linear matrix equations of the above three types are encountered in many applications of control and engineering
problems \cite{Zhou2011,Chiang2012}. To solve these matrix equations, the first idea is to transform  these matrix equations into the standard Stein matrix equation
\begin{align*}
\mathcal{X}=\mathcal{A} \mathcal{X} \mathcal{B}+\mathcal{C},
\end{align*}
where $\mathcal{X}$ is an unknown matrix to be determined, and $\mathcal{A}=Af(B),\,\mathcal{B}=f(A)B,\,\mathcal{C}=C+Af(C)B$ in type 1 and type 2, and $\mathcal{A}=Af(A),\,\mathcal{B}=f(B)B,\,\mathcal{C}=C+Af(C)B$ in type 3.
Quoted from \cite{Zhou2011}, the authors give a beauty formulation
\begin{align}\label{zhou}
{X}=\dfrac{1}{2}\left(\mathcal{X}+Af(\mathcal{X})B+C\right)
\end{align}
to finding a solution $X$ of the original Eq.~\eqref{eq} through the solution $\mathcal{X}$ of the Stein matrix equation. Moreover, the general solution of the original matrix equation can be obtained from the general solution of the Stein matrix equation.

We give an insightful observation about the properties of the matrix operator $f$ in the above three types, the formulation \eqref{zhou} arises in the periodically and (anti-)linearity of the matrix operator $f$. It is interesting to ask whether the formulation \eqref{zhou} can be established in a generalized case? Roughly speaking, if the matrix operator $f$ is a (anti-)linear and period-$n$ matrix operator. We are going to extend the recent results and hope to draw further attention to derive a formulation like  \eqref{zhou} in the case of \cite{Zhou2011} and thus find the general solution of Eq.~\eqref{eq}.

To advance our research we organize this paper as follows. Inspired by existing work\cite{Zhou2011}, after a subsequent computations and observations, the transformation between the general solution of new Stein matrix equation and the general solution of original matrix equation is obtained in Section 2. Beginning in Section 3, we formulate the necessary conditions for the existence of the solution of \eqref{eq} directly by means of the conditions for the existence of the solution of new Stein matrix equation. Also, a close-form of the unique solution of Eq.~\eqref{eq} is shown under a regularity assumption on the coefficients $A$, $B$ and $C$. Two expressions of our formulation are received by way of apply for our result to some equivalent equations of Eq.~\eqref{eq}. In the remainder of this paper, we offer three examples to demonstrate our consequence in Section~5. In Section~6, we provide a concluding summary.

\section{Formulae}
In order to formalize our discussion, let us start with some notations. $f^{(i)}$ denote the composition of $f$ with itself $i$ times and  $f^{(0)}\equiv I$ is the identity map. $A\otimes B$ be the Kronecker product of matrices $A$ and $B$. $\mbox{vec}(A)$ stacks the columns of $A$ into a column vector and the vector operator $\mbox{vec}(.)$ is invertible. For a matrix $A$, $\sigma(A)$ is the set of all eigenvalues of $A$ and $\rho(A)$ is denoted by the spectral radius of $A$. For the sake of convenience, we without loss of generality assume that $f$ is linear through this paper.

First, for ease of notations. Let the specific operator $G$ involving a integer parameter $i\geq 1$ be defined as
$$G_i(X)=\prod_{j=1}^i f^{(j-1)}(A) \left[f^{(i)}(X)\right]\prod_{j=1}^i f^{(i-j)}(B)$$
for arbitrary $m$-square matrix $X$. In particular, we also define $G_0(X)=X$. It is easy to see that $G_i(G_j(X))=G_j(G_i(X))=G_{i+j}(X)$ for any two positive integers $i$ and $j$.
On account of the periodically of the operator $f$, Eq.~\eqref{eq} can be transformed into a auxiliary  Stein matrix equation with respect to $\mathcal{X}$ duo to a composition of $n$ copies of the right hand side of Eq.~\eqref{eq},
\begin{align}\label{Stein}
\mathcal{X}= \mathcal{A}\mathcal{X}\mathcal{B}+\mathcal{C},
\end{align}
where the coefficient matrices $\mathcal{A},\mathcal{B}$ and $\mathcal{C}$ are defined as
\begin{subequations}\label{AABBCC}
\begin{align}
\mathcal{A}&=\mathcal{A}_n=\prod_{i=1}^n f^{(i-1)}(A),\label{AA} \\
\mathcal{B}&=\mathcal{B}_n=\prod_{i=1}^n f^{(n-i)}(B),\label{BB}\\
\mathcal{C}&=\mathcal{C}_n=\sum\limits_{i=1}^n G_{i-1}(C).\label{CC}
\end{align}
\end{subequations}
It is obvious that $X$ solve the Stein matrix equation~\eqref{Stein} if $X$ is a solution of the linear matrix equation~\eqref{eq}. However, the converse does not necessarily hold. Let $A = -1$, $B = 1$ and matrix operator $f$ be the identity map in \eqref{eq}. It is clear that, the scalar equation $X=-X+C$ has a unique solution $X=\dfrac{C}{2}$. But, the Eq.~\eqref{Stein} becomes to the identity equation $X=X$, which has infinite solutions.

In this section, we are interested in the relationship between the set of the general solution of Eq.~\eqref{eq} and the set of general solution of Eq.~\eqref{Stein}. To simplify the discussion. Let $\mathcal{S}_1$ be the set of general solutions of Eq.~\eqref{eq} and $\mathcal{S}_2$ be the set of general solutions of Eq.~\eqref{Stein}. From the above conclusion we have
\begin{align*}
\mathcal{S}_1\subseteq \mathcal{S}_2.
\end{align*}
Consequently, our motivation in this study is to find a map from $\mathcal{S}_2$ to $\mathcal{S}_1$. In order to reach the goal. Next, we are going to survey the deep structure of Eq.~\eqref{Stein}. With the notation of the matrix operator $G$, the stand Stein matrix equation~\eqref{Stein} can be written as the follows form
\begin{align}\label{GiX}
\mathcal{X}=G_n(\mathcal{X})+\sum\limits_{i=0}^{n-1} G_{i}(C).
\end{align}

We observe the summation of $G_k(\mathcal{X})$ from $k=0$ to $k=n-1$,
\begin{align}\label{SG}
&\sum\limits_{k=0}^{n-1} G_k(\mathcal{X})=\prod_{j=1}^n f^{(j-1)}(A) \sum\limits_{k=0}^{n-1} G_k(\mathcal{X})\prod_{j=1}^n f^{(n-j)}(B)
+\sum\limits_{i=0}^{n-1}\sum\limits_{k=0}^{n-1} G_{k+i}(C).
\end{align}
Denote the last team of \eqref{SG} by $S_n$ and it summate total $n^2$ teams.
we collect the sets with the same sum of indexs $i$ and $k$,
\begin{align*}
S_n&=\sum\limits_{i=0}^{n-1}\sum\limits_{k=0}^{n-1} G_{k+i}(C)=\left (\sum\limits_{k+i=0}+\cdots+\sum\limits_{k+i=2n-2}\right)G_{k+i}(C)\\
&=\sum\limits_{k=0}^{n-2}(k+1) G_{k}(C)+nG_{n-1}(C)+\sum\limits_{k=n}^{2n-2}(2n-k-1)G_{k}(C)\\
&=\sum\limits_{k=0}^{n-1}(k+1) G_{k}(C)+\sum\limits_{k=0}^{n-2}(n-k-1)G_{k+n}(C).
\end{align*}
Note that $G_{k+n}(C)=\mathcal{A} G_{k}(C)\mathcal{B}$, then,
\begin{align*}
 \sum\limits_{k=0}^{n-1}\left(  G_k(\mathcal{X})+(n-k-1)G_k(C) \right)=\mathcal{A}\sum\limits_{k=0}^{n-1}\left(  G_k(\mathcal{X})+(n-k-1)G_k(C) \right)\mathcal{B}+n\mathcal{C}.
 \end{align*}
If we define the matrix operator $F_{f}:\mathbb{C}^{m\times m}\rightarrow \mathbb{C}^{m\times m}$ associate with the matrix operator $f$ and Eq.~\eqref{eq} by
\begin{subequations}\label{3Ff}
\begin{align}\label{Ff}
F_{f}(\mathcal{X})=\dfrac{1}{n}\left( \sum\limits_{i=0}^{n-1} \left(G_i(\mathcal{X})+(n-i-1)G_i(C)\right) \right),
\end{align}
where $\mathcal{X}\in\mathbb{C}^{m\times m}$. As a consequence, it is apparent that the matrix $F_{f}(\mathcal{X})$ also solve Eq.~\eqref{Stein} if  $\mathcal{X}$ solve Eq.~\eqref{Stein}.
An alternative way to present this formula, the formula of \eqref{Ff} can be written as the following,
\begin{align}\label{Ff1}
F_{f}(\mathcal{X})=\dfrac{1}{n}\left( \sum\limits_{i=0}^{n-1} (G_i(\mathcal{X})+\sum\limits_{j=0}^{i-1} G_j(C)) \right).
\end{align}
\end{subequations}
The matrix operator $F_f$ can also be viewed as a meaningful decomposition,
let the operator $\mathcal{O}$ be defined as the right hand side of the Eq.~\eqref{eq}
\begin{align*}
\mathcal{O}(\mathcal{X})=Af(\mathcal{X})B+C,
\end{align*}
hereafter ${F}_{f}(\mathcal{X})$ can be regard as the algebraic average of the first $n$ elements of the forward orbit of $\mathcal{O}$ start with $\mathcal{X}$
\begin{align*}
\mathcal{X},\mathcal{O}(\mathcal{X}),\cdots,\mathcal{O}^{(n-1)}(\mathcal{X}).
\end{align*}
Now, we firstly give the following important feature for $F_{f}(\mathcal{X})$ with the specific matrix $\mathcal{X}$.
\begin{Proposition}\label{Pro1}
\begin{itemize}
\item[1.]
Suppose that $\mathcal{X}$ is a solution of the Stein matrix equation~\eqref{Stein}, then $F^{(k)}_{f}(\mathcal{X})$ solve Eq.~\eqref{Stein} for any positive integer $k$.
\item[2.]
Suppose that $X$ is a solution of the linear matrix equation~\eqref{eq}, then $F_{f}(X)=X$.
\end{itemize}
\end{Proposition}
\begin{proof}
For the first part, from the above discussion, $F_{f}(\mathcal{X})$ solve Eq.~\eqref{Stein}. And repeat this conclusion, $F^{(2)}_{f}(\mathcal{X}),\cdots,F^{(k)}_{f}(\mathcal{X})$ both solve Eq.~\eqref{Stein} for each positive integer $k$.

For the second part and arbitrary positive integer $i$, since
\begin{align*}
X=\mathcal{O}^{(i)}(X)=G_i(X)+\sum\limits_{j=0}^{i-1} G_j(C),
\end{align*}
together with \eqref{Ff1}, this completes the proof.
\end{proof}
The first part of Proposition~\ref{Pro1} state that there is a family of solutions $\{ F^{(k)}_{f}(\mathcal{X}) \}$ of Eq.~\eqref{Stein}. The following property tell us that the number of elements of $\{ F^{(k)}_{f}(\mathcal{X}) \}$ is at most 2.
\begin{Proposition}
Suppose that $\mathcal{X}$ is a solution of the linear matrix equation~\eqref{Stein}, then
$F_{f}(\mathcal{X})$ is a fixed-point of $F_{f}$, i.e.,  $F^{(2)}_{f}(\mathcal{X})=F_{f}(\mathcal{X})$.
\end{Proposition}
\begin{proof}
We claim the following identical equation,
\begin{align}\label{fix1}
\sum\limits_{i=0}^{n-1}G_i(F_{f}(\mathcal{X}))=\sum\limits_{i=0}^{n-1} G_i(\mathcal{X}).
\end{align}
Since $\mathcal{X}$ solve the linear matrix equation~\eqref{Stein}, we have
\begin{align*}
G_{i}(\mathcal{X})=G_{i}\left(G_n(\mathcal{X})+\sum\limits_{j=0}^{n-1} G_{j}(C)\right)=G_{i+n}(\mathcal{X})+\sum\limits_{j=0}^{n-1} G_{i+j}(C)
\end{align*}
for $j=0,1,\cdots n-1$. Hence, the left hand side from \eqref{fix1} can be written
\begin{align*}
&\sum\limits_{i=0}^{n-1}G_i(F_{f}(\mathcal{X}))=\dfrac{1}{n}\left(\sum\limits_{i=0}^{n-1}\sum\limits_{j=0}^{n-1} G_{i+j}(\mathcal{X})+(n-j-1)G_{i+j}(C)\right)\\
&=\dfrac{1}{n}\left(\sum\limits_{i=0}^{n-1} (i+1)G_{i}(\mathcal{X})+\sum\limits_{i=0}^{n-1}(n-i-1)\left(G_{i+n}(\mathcal{X})+\sum\limits_{j=0}^{n-1}G_{i+j}(C)\right)\right)\\
&=\sum\limits_{i=0}^{n-1} G_i(\mathcal{X})\\
\end{align*}
is exactly equal to the right hand side from \eqref{fix1}.
Eventually, the result now follows from the expansion of $F_{f}(\mathcal{X})$ and $F^{(2)}_{f}(\mathcal{X})$.

Actually, from the part~1 of Proposition~\ref{Pro1}, and $F_{f}(\mathcal{X})$ is a solution of Eq.~\eqref{eq}. Hence, the equality $F^{(2)}_{f}(\mathcal{X})=F_{f}(\mathcal{X})$ immediately come from the part~2 of Proposition~\ref{Pro1}.
\end{proof}
The most important role of the matrix operator $F_f$  is stated as follows:
\begin{Theorem}\label{solve eq}
Let $\mathcal{X}$ be a solution of the Stein matrix equation~\eqref{Stein}, then $F_{f}(\mathcal{X})$ solve the original linear matrix equation~\eqref{eq} for $1\leq k\leq n-1$.
\end{Theorem}
\begin{proof}
By directly calculate, since $\mathcal{O}^{(i+1)}(\mathcal{X})=Af(\mathcal{O}^{(i)}(\mathcal{X}))B+C$ for arbitrary nonnegative integer $i$. We have
\begin{align*}
&F_{f}(\mathcal{X})-Af(F_{f}(\mathcal{X}))B =
\dfrac{1}{n}\sum\limits_{i=0}^{n-1}\mathcal{O}^{(i)}(\mathcal{X})
-\dfrac{1}{n} \sum\limits_{i=1}^{n}\left(\mathcal{O}^{(i)}(\mathcal{X})-C\right)\\
&=\dfrac{1}{n}\left(\mathcal{X}-\mathcal{O}^{(n)}(\mathcal{X})+nC\right)=C,
\end{align*}
which completes the proof.
\end{proof}
Now we have enough experience to analyze the structure of the set of $\mathcal{S}_1$, the set of $\mathcal{S}_2$ can be decomposed into two sets
\begin{align*}
\mathcal{S}_2=\mathcal{S}_1 \cup (\mathcal{S}_2 \setminus \mathcal{S}_1).
\end{align*}
As mentioned above, $X\in\mathcal{S}_1$ implies that $F_{f}(X)=X\in\mathcal{S}_1$ and $\mathcal{X}\in\mathcal{S}_2 \setminus \mathcal{S}_1$ implies that $F_{f}(\mathcal{X})\in\mathcal{S}_1$. The first conclusion is $\mathcal{S}_1=F_{f}(\mathcal{S}_2)$, namely, the operator $F_{f}:\mathcal{S}_2\rightarrow\mathcal{S}_1$ is a surjective function. And, the general solution of Eq.~\eqref{eq} can be found from the general solution of Eq.~\eqref{Stein}.
Second, if we restrict the domain $\mathcal{S}_2$ to $\mathcal{S}_1$, then $F_{f}:\mathcal{S}_1\rightarrow\mathcal{S}_1$ is an identity map. We figure out the operator $F_{f}:\mathcal{S}_2\rightarrow\mathcal{S}_1$ is an injective function if and only if $F_{f}(\mathcal{X})=\mathcal{X}$ for all $\mathcal{X}\in\mathcal{S}_2$ and is equivalent to $\mathcal{S}_1=\mathcal{S}_2$. We answer the question as mentioned before in the beginning of this section.
\section{Solvability conditions of Eq.~\eqref{eq}}
In this section, the necessary and sufficient conditions for the existence of the solution of Eq.~\eqref{eq} are shown. We also give a necessary  conditions for the uniquely solvable of the solution of Eq.~\eqref{eq}, and provide the close-form of solution of Eq.~\eqref{eq} whenever Eq.~\eqref{eq} is uniquely solvable. Some iterative methods for finding the unique solution of Eq.~\eqref{eq} are suggested.

First of all, we summarize the above result in Section~2 and an immediate consequence of Theorem~\ref{solve eq} is following.
\begin{Theorem}\label{main}
The linear matrix equation~\eqref{eq} is solvable if and only if the Stein matrix equation~\eqref{Stein} is solvable. Furthermore, Eq.~\eqref{eq} has a unique solution $X$ if Eq.~\eqref{Stein} has a unique solution $\mathcal{X}$.
\end{Theorem}

Before demonstrating the unique solvability conditions of Eq.~\eqref{eq},
 we need to define that two sets $S_1$ and $S_2$ of complex numbers are said to be \emph{reciprocal free} if and only if $1/\lambda\not\in S_1$ whenever $\lambda\in S_2$. This definition also regards $0$ and $\infty$ as reciprocals of each other. We also notice that a single set $S$ is reciprocal free if and only if $1/\lambda\not\in S$ whenever $\lambda\in S$. By making use of Theorem~\ref{solve eq} and apply Kronecker product to Eq.~\eqref{Stein}, we can reveal the following outcome.
\begin{Theorem}\label{Coro}
If matrices $\mathcal{A}$ and $\mathcal{B}$ in Eq.~\eqref{Stein} such that the set of $\sigma(\mathcal{A})$ and the set of $\sigma(\mathcal{B})$ are reciprocal free, then Eq.~\eqref{eq} is uniquely solvable. Moreover, the unique solution of Eq.~\eqref{eq} is
\begin{align*}
X=\dfrac{1}{n}\left( \sum\limits_{i=0}^{n-1}\prod_{j=1}^i f^{(j-1)}(A) \left[f^{(i)}\left(\mathcal{X}\right)+(n-i-1)f^{(i)}(C) \right]\prod_{j=1}^i f^{(i-j)}(B) \right),
\end{align*}
where $\mathcal{X}={\rm{vec}}^{-1}((I_{n^2}-\mathcal{B}^\top\otimes\mathcal{A})^{-1}{\rm{vec}}(\mathcal{C}))$ is the unique solution of Eq.~\eqref{Stein}.
\end{Theorem}

Now we should pay attention to finding the close-form of the solution $X$ of Eq.~\eqref{eq}. For arbitrary solution $\mathcal{X}$ of Eq.~\eqref{Stein}, substituting Eq.~\eqref{GiX} into the representation of $F_{f}$ \eqref{Ff}
\begin{align*}
F_{f}(\mathcal{X})-\sum\limits_{i=0}^{n-1} G_i(C)=G_n(F_{f}(\mathcal{X})).
\end{align*}
Performing the same procedure, we have
\begin{align*}
F_{f}(\mathcal{X})-\sum\limits_{j=0}^{k-1}\sum\limits_{i=0}^{n-1} G_{i+jn}(C)=G_{kn}(F_{f}(\mathcal{X})).
\end{align*}
 Under some reasonable requirements of matrices $\mathcal{A}$ and $\mathcal{B}$, such as $\rho(\mathcal{A})\rho(\mathcal{B})<1$. Then $G_{kn}(F_{f}(\mathcal{X}))=\mathcal{A}^k F_{f}(\mathcal{X}) \mathcal{B}^k$ tends to zero as $k$ approach to infinity.  We can get a close-form of $F_{f}(\mathcal{X})$ when $\mathcal{X}$ is a solution of Eq.~\eqref{Stein}. More precisely, it can be obtained a numerical method to compute $F_{f}(\mathcal{X})$.
\begin{Theorem}\label{close}
Under the mile condition
\begin{align}\label{cond}
\rho\equiv\rho(\mathcal{A})\rho(\mathcal{B})<1.
\end{align}
Then Eq.~\eqref{Stein} is uniquely solvable and therefore Eq.~\eqref{eq} is also uniquely solvable. Furthermore,
the unique solution $X$ of Eq.~\eqref{eq} satisfies
\begin{align*}
\limsup\limits_{k\rightarrow\infty}\sqrt[k]{\| X-\sum\limits_{j=0}^{k-1}\mathcal{A}^j \left[\sum\limits_{i=0}^{n-1} G_{i}(C)\right] \mathcal{B}^j\|}\leq \rho.
\end{align*}
\end{Theorem}

\begin{Remark}
The close-form of $X$ in Theorem~\ref{close} is called the Smith iteration. As a matter of a fact, based on Eq.~\eqref{GiX}, if the assumption \eqref{cond} holds, Zhou et al. have proposed the following Smith-type iterations \cite{Zhou2009}.
\begin{itemize}
\item[1.]
Since $X-\sum\limits_{j=0}^{k-1}\sum\limits_{i=0}^{n-1} G_{i+jn}(C)=G_{kn}(X)$, we define $X_k=\sum\limits_{j=0}^{k-1} \mathcal{A}^j\sum\limits_{i=0}^{n-1} G_{i}(C) \mathcal{B}^j$ and $X_k$ can be designed as the following Smith iteration,
\begin{align*}
X_0 &= 0,\\
X_{k+1} &= \mathcal{A} X_{k} \mathcal{B}+\mathcal{C}.
\end{align*}
The convergence rate can be shown
\begin{align*}
\limsup\limits_{k\rightarrow\infty}\sqrt[k]{\| X-X_k|}\leq \rho.
\end{align*}
\item[2.]
Given a prescribed positive integer $\ell$, we apply Smith iteration to the Stein matrix equation, which is equivalent to Eq.~\eqref{eq},
\begin{align*}
X= \mathcal{A}_{\ell n}X\mathcal{B}_{\ell n}+\mathcal{C}_{\ell n}.
\end{align*}
We define $X_k=\sum\limits_{j=0}^{k-1}\mathcal{A}^{\ell j} \left[\sum\limits_{i=0}^{\ell n-1} G_{i}(C)\right] \mathcal{B}^{\ell j}$ and $X_k$ can be designed as the following Smith($\ell$) iteration.
\begin{align*}
X_0 &= 0,\\
X_{k+1} &= \mathcal{A}^\ell X_{k} \mathcal{B}^\ell+\sum\limits_{i=0}^{\ell-1}\mathcal{A}^i \mathcal{C}\mathcal{B}^i.
\end{align*}
The convergence rate can be shown
\begin{align*}
\limsup\limits_{k\rightarrow\infty}\sqrt[k]{\| X-X_k\|}\leq \rho^\ell.
\end{align*}
\item[3.]
Given a prescribed positive integer $r$, we observe
\begin{align*}
X&=\mathcal{A}^r X \mathcal{B}^r+\sum\limits_{i_1=0}^{r-1}\mathcal{A}^{i_1}\mathcal{C}\mathcal{B}^{i_1}  =(\mathcal{A}^{r})^r X (\mathcal{B}^{r})^r+\sum\limits_{i_2=0}^{r-1}\mathcal{A}^{ri_2}\left[\sum\limits_{i_1=0}^{r-1} \mathcal{A}^{i_2}\mathcal{C}\mathcal{B}^{i_2}\right]\mathcal{B}^{ri_2}\\
&=\mathcal{A}^{r^2} X \mathcal{B}^{r^2}+\sum\limits_{i_1,i_2=0}^{r-1} \mathcal{A}^{i_1+ri_2}\mathcal{C}\mathcal{B}^{i_1+ri_2}=\cdots\\
&=\mathcal{A}^{r^k} X \mathcal{B}^{r^k}+\sum\limits_{i_1,i_2,\cdots,i_k=0}^{r-1} \mathcal{A}^{i_1+ri_2+\cdots+r^{k-1}r_k}\mathcal{C}\mathcal{B}^{i_1+ri_2+\cdots+r^{k-1}r_k}.
\end{align*}
We define $X_k=\sum\limits_{i_1,i_2,\cdots,i_k=0}^{r-1} \mathcal{A}^{i_1+ri_2+\cdots+r^{k-1}r_k}\mathcal{C}\mathcal{B}^{i_1+ri_2+\cdots+r^{k-1}r_k}$ and $X_k$ can be designed as the following $r$-Smith iteration,
\begin{align*}
A_0&=\mathcal{A},\,B_0=\mathcal{B},\,X_0 =\mathcal{C},\\
A_{k+1}&=A_k^r,\,B_{k+1}=B_k^r,\,  X_{k+1} = \sum\limits_{i=0}^{r-1}\mathcal{A}_k^i \mathcal{C}\mathcal{B}_k^i.
\end{align*}
The convergence rate can be shown
\begin{align*}
\limsup\limits_{k\rightarrow\infty}\sqrt[r^k]{\| X-X_k\|}\leq \rho.
\end{align*}
This is so-call $r$-Smith iteration.
\end{itemize}
Smith-type iterations have attracted much interests in many papers, one can see \cite{Zhou2011,Zhou2009}  and the references therein.
\end{Remark}
\begin{Remark}
If $f$ is anti-linear, we only need replace the coefficients of Eq.~\eqref{Stein} $\mathcal{A}$, $\mathcal{B}$ and $\mathcal{C}$ in \eqref{AABBCC} and the formulation \eqref{3Ff} with the correspondingly following.
\begin{itemize}
\item[a.]
If $n$ is even,
\begin{align*}
\mathcal{A}&=\prod_{i=1}^{\frac{n}{2}} f^{(2i-2)}(A)f^{(2i-1)}(B),\quad\mathcal{B}=\prod_{i=1}^{\frac{n}{2}} f^{(n-2i+1)}(A)f^{(n-2i)}(B),\\
\mathcal{C}&=\sum\limits_{i=1}^{\frac{n}{2}} \left(\prod_{j=1}^{i-1} f^{(2j-2)}(A)f^{(2j-1)}(B)\right) \left[ f^{(2i-2)}(C)+f^{(2i-2)}(A) f^{(2i-1)}(C) f^{(2i-2)}(B) \right] \\
 &\left( \prod_{j=1}^{i-1} f^{(2j-1)}(A)f^{(2j-2)}(B)\right),
\end{align*}
and
\begin{align*}
&F_{f}(\mathcal{X})=\dfrac{1}{n}\left( \sum\limits_{i=0}^{\frac{n-2}{2}}\prod_{j=1}^i f^{(2j-2)}(A)f^{(2j-1)}(B) \left[f^{(2i)}(\mathcal{X})+(n-2i-1)f^{(2i)}(C)\right.  \right. \\
&\left.+f^{(2i)}(A) [f^{(2i+1)}(\mathcal{X})+(n-2i-2)f^{(2i+1)}(C)] f^{(2i)}(B)\right]\left.\prod_{j=1}^i f^{(2j-1)}(A)f^{(2j-2)}(B) \right).
\end{align*}
\item[b.]If $n$ is odd,
\begin{align*}
\mathcal{A}&=\left(\prod_{i=1}^{\frac{n-1}{2}} f^{(2i-2)}(A)f^{(2i-1)}(B)\right)f^{(n-1)}(A),\quad\mathcal{B}=f^{(n-1)}(B)\left(\prod_{i=1}^{\frac{n-1}{2}} f^{(n-2i+1)}(A)f^{(n-2i)}(B)\right),\\
\mathcal{C}&=\left(\prod_{j=1}^{\frac{n-1}{2}} f^{(2j-2)}(A)f^{(2j-1)}(B)\right)f^{(n-1)}(C)\left(\prod_{j=1}^{\frac{n-1}{2}} f^{(2j-1)}(A)f^{(2j-2)}(B)\right)+ \\
 &\sum\limits_{i=1}^{\frac{n-1}{2}}\left(\prod_{j=1}^{i-1} f^{(2j-2)}(A)f^{(2j-1)}(B)\right)\left[ f^{(2i-2)}(C)+f^{(2i-2)}(A) f^{(2i-1)}(C) f^{(2i-2)}(B) \right]\\
 &\left( \prod_{j=1}^{i-1} f^{(2j-1)}(A)f^{(2j-2)}(B)\right),
\end{align*}
and
\begin{align*}
&F_{f}(\mathcal{X})=\dfrac{1}{n}\left( \left(\prod_{j=1}^{\frac{n-1}{2}} f^{(2j-2)}(A)f^{(2j-1)}(B)\right) f^{(n-1)}(\mathcal{X})\left( \prod_{j=1}^{\frac{n-1}{2}} f^{(2j-1)}(A)f^{(2j-2)}(B)\right)\right.\\
&+\sum\limits_{i=0}^{\frac{n-3}{2}}\prod_{j=1}^i f^{(2j-2)}(A)f^{(2j-1)}(B) \left[f^{(2i)}(\mathcal{X})+(n-2i-1)f^{(2i)}(C)\right.\left.\right. \\
&\left.+f^{(2i)}(A) [f^{(2i+1)}(\mathcal{X})+(n-2i-2)f^{(2i+1)}(C)] f^{(2i)}(B)\right]\left.\prod_{j=1}^i f^{(2j-1)}(A)f^{(2j-2)}(B) \right).
\end{align*}
\end{itemize}
\end{Remark}
\section{Some representations of $F_{f}$}
In this section, we present two equivalent forms of the matrix operator $F_f$ as defined in \eqref{3Ff} under some reasonable speculations. First, since period-$n$ implies period-$kn$ for any positive integer $k$. After a composition of $kn$ copies of $\mathcal{O}$ of the right hand side of Eq.~\eqref{eq}, we still have a Stein matrix equation
\begin{align}\label{knStein}
Y= \mathcal{A}_{kn}Y\mathcal{B}_{kn}+\mathcal{C}_{kn}.
\end{align}
Note that if $\mathcal{X}$ solve the Stein matrix equation~\eqref{Stein}, then $\mathcal{X}$ also satisfies the new Stein matrix equation~\eqref{knStein}.
The matrix $Y$ solve \eqref{knStein} and the correspondingly $F_{f}(Y)$ is
\begin{align}\label{prove1}
F_{f}(Y)&=\dfrac{1}{kn}\left(\sum\limits_{i=0}^{kn-1} G_i(Y)+(kn-i-1)G_i(C) \right).
\end{align}
It is nature to ask whether the new matrix operator $F_{f}(\mathcal{X})$ associate with Eq.~\eqref{knStein} in \eqref{prove1} is equal to the original matrix operator $F_{f}(\mathcal{X})$ associate with Eq.~\eqref{Stein} in \eqref{3Ff} whenever $\mathcal{X}$ solve Eq.~\eqref{Stein}? The answer is true, we have the following theorem.
\begin{Theorem}
Let the matrix operator $F_f$ be defined as in \eqref{prove1}, then the value of
$F_{f}(\mathcal{X})$  associate with Eq.~\eqref{knStein} is independent of $k$.
\end{Theorem}
\begin{proof}
It is clear that
\begin{align*}
\mathcal{A}_{kn}=\mathcal{A}_{}^k,\,\mathcal{B}_{kn}=\mathcal{B}_{}^k,
\end{align*}
and
\begin{align*}
G_{i+kn}(X)=G_i(\mathcal{A}_{}^k X\mathcal{B}_{}^k)=\mathcal{A}^{k} G_i(X) \mathcal{B}^{k},
\end{align*}
for any matrix $X$. Taking the matrix operator $G_i$ to both sides of Eq.~\eqref{Stein}, it follows that
\begin{align*}
G_{i}(\mathcal{X})=\mathcal{A} G_i(\mathcal{X}) \mathcal{B}+\sum\limits_{j=0}^{n-1}  G_{i+j}(C),
\end{align*}
we shall prove the equality \eqref{prove1} by using mathematical induction, where $\mathcal{X}$ is a solution of Eq.~\eqref{Stein}. To avoid confusion notations we rewritten the matrix operator $F_{f}$ in \eqref{prove1} as $F_{f,k}$.  For $k=1$ it is trivial. Assume that the conclusion of \eqref{prove1} is true for $k=k_0-1$, where $k_0$ is a positive integer. Then, together with Theorem~\ref{solve eq},
\begin{align*}
&F_{f,k_0}(\mathcal{X})=\dfrac{1}{k_0n}\left( \sum\limits_{i=0}^{n-1}\left( G_i(\mathcal{X})+(n-i-1)G_i(C)\right)+ \sum\limits_{i=0}^{(k_0-1)n-1}\mathcal{A}\left( G_i(\mathcal{X})\right.\right.\\
&\left.+((k_0-1)n-i-1)G_i(C)\right)\mathcal{B}
\left.+(k_0-1)\sum\limits_{i=0}^{n-1}G_i(C) \right)\\
&=\dfrac{1}{k_0}\left(F_{f}(\mathcal{X})+(k_0-1)F_{f}(\mathcal{X})\right)=F_{f}(\mathcal{X}).
\end{align*}
\end{proof}
Next, we consider a new linear matrix equation associate with a positive integer $k$ be defined as follows
\begin{align}\label{k1n}
Y=\left(\mathcal{A}^{k}A\right) f(Y)\left( B\mathcal{B}^{k}\right)+\mathcal{A}^{k} C\mathcal{B}^{k}+\sum\limits_{j=0}^{k-1}\mathcal{A}^{j} \mathcal{C}\mathcal{B}^{j},
\end{align}
which is obtained after a composition of $kn+1$ copies of $\mathcal{O}$ of the right hand side of Eq.~\eqref{eq}. On the other hand, it is straightforward to show that
\begin{align*}
\prod_{j=1}^i f^{(j-1)}(\mathcal{A}^{k}A)=\mathcal{A}^{ik} \prod_{j=1}^i f^{(j-1)}(A),\,\prod_{j=1}^i f^{(j-1)}(B\mathcal{B}^{k})=\prod_{j=1}^i f^{(j-1)}(B)\mathcal{B}^{ik}.
\end{align*}
Applying formulation \eqref{Ff} to the new Eq.~\eqref{k1n} with the substitution $(\mathcal{A},\mathcal{B},\mathcal{C})\rightarrow (\mathcal{A}^{k}A,B\mathcal{B}^{k},\mathcal{A}^{k} C\mathcal{B}^{k}+\sum\limits_{j=0}^{k-1}\mathcal{A}^{j} \mathcal{C}\mathcal{B}^{j})$, we have
\begin{align}\label{prove2}
F_{f}(Y)&=\dfrac{1}{n}\left(\sum\limits_{i=0}^{n-1} \mathcal{A}^{ik} G_i(Y) \mathcal{B}^{ik}+(n-i-1)\right.\nonumber\\
&\left.\left[\mathcal{A}^{k(i+1)}G_{i}(C)\mathcal{B}^{k(i+1)}+\sum\limits_{j=0}^{k-1} \mathcal{A}^{ki+j}G_{i}(\mathcal{C})\mathcal{B}^{ki+j}\right] \right)\nonumber\\
&=\dfrac{1}{n}\left(\sum\limits_{i=0}^{n-1} \mathcal{A}^{ik}\left[ G_i(Y) +(n-i-1)(\mathcal{A}^{k}G_{i}(C)\mathcal{B}^{k}+\sum\limits_{j=0}^{k-1} \mathcal{A}^{j}G_{i}(\mathcal{C})\mathcal{B}^{j})\right]\mathcal{B}^{ik}\right).
\end{align}
In the light of the characteristic of the original matrix operator $F_f$ in \eqref{3Ff}, it is not surprising that the following
conclusions are appeared.
\begin{Proposition}\label{ProY}
\begin{itemize}
\item[1.]
Suppose that $Y$ is a solution of the Stein matrix equation~\eqref{Stein}, then $F^{(j)}_{f}(Y)$ solve Eq.~\eqref{Stein} for any positive integer $j$.
\item[2.]
Suppose that $Y$ is a solution of the linear matrix equation~\eqref{eq}, then $F_{f}(Y)=Y$.
\end{itemize}
\end{Proposition}
\begin{proof}
For each $0\leq i\leq n$, let
\begin{align*}
\Delta_i=\mathcal{A}^{k}G_{i}(C)\mathcal{B}^{k}+\sum\limits_{j=0}^{k-1} \mathcal{A}^{j}G_{i}(\mathcal{C})\mathcal{B}^{j},
\end{align*}
if $Y$ satisfies the Eq.~\eqref{Stein} and therefore also solve Eq.~\eqref{knStein}, thus
\begin{align*}
\Delta_i=G_i(Y)-\mathcal{A}^{k}G_{i}Y\mathcal{B}^{k}+\mathcal{A}^{k}G_{i}(C)\mathcal{B}^{k},
\end{align*}
and it is easy to see that $G_i(Y)-G_{i+1}(Y)=G_i(C)$. Then,
\begin{align*}
&F_{f}(Y)-A f(F_{f}(Y)) B=\dfrac{1}{n}\left(\sum\limits_{i=0}^{n-1}  (n-i)G_{i+ikn}(Y)-(n-i-1)\left(G_{i+(i+1)kn}(Y)-G_{i+(i+1)kn}(C)\right)\right)\\
&-\dfrac{1}{n}\left(\sum\limits_{i=0}^{n-1}  (n-i)G_{i+1+ikn}(Y)-(n-i-1)\left(G_{i+1+(i+1)kn}(Y)-G_{i+1+(i+1)kn}(C)\right)\right)\\
&=\dfrac{1}{n}\left(\sum\limits_{i=0}^{n-1} (n-i)G_{i+ikn}(C)-(n-i-1)\left[G_{i+(i+1)kn}(C)-G_{i+(i+1)kn}(C)+G_{i+1+(i+1)kn}(C)\right]\right)\\
&=\dfrac{1}{n}\left(nG_0(C)\right)=C.
\end{align*}
This completes the first part. For the second part, if $Y$ satisfies the Eq.~\eqref{eq} and therefore also solve Eq.~\eqref{k1n}, an obvious argument gives
\begin{align*}
\Delta_i=G_i(Y)-\mathcal{A}^{k}G_{i+1}(Y)\mathcal{B}^{k}.
\end{align*}
Then,
\begin{align*}
F_{f}(Y)&=\dfrac{1}{n}\left(\sum\limits_{i=0}^{n-1}  G_{i+ikn}(Y)+(n-i-1)\left(G_{i+ikn}(Y)-G_{i+1+(i+1)kn}(Y)\right)\right)\\
&=\dfrac{1}{n}\left(\sum\limits_{i=0}^{n-1}  (n-i)G_{i+ikn}(Y)-(n-i-1)G_{i+1+(i+1)kn}(Y)\right)\\
&=Y,
\end{align*}
we have thus proved the second part.
\end{proof}
In summary of this section, it seem that two representations of $F_f$ are theoretically interesting.
However, compare the computational complexity of these formulations \eqref{Ff}, \eqref{prove1} and \eqref{prove2}.
 In practice, we usually evaluate the value of $F_f(\mathcal{X})$ by employing the original formulation \eqref{3Ff}.


\section{Several examples}
We shall now give three examples to demonstrate some results of the previous sections. In the first example, we explain our formulation with a linear and period-$n$ matrix operator $f$. The second example provide the necessary and sufficient conditions for the uniquely existence of the solution of a particular matrix equation. Finally, we develop some results obtained from the previous sections can be extended to a matrix equation having more terms on the right hand side of Eq.~\eqref{eq} in the last example.
\begin{example}
$P$ is a permutation matrix if each row of $P$ and each column of $P$ possesses
one 1 and zeros otherwise. Let $P\in\mathbb{R}^{m \times m}$ be the primary permutation matrix, that is,
\begin{align*}
P=\bb 0 & 1& 0 &\ldots&0\\
0 & 0& 1 &\ldots&0\\
\vdots & \vdots& \vdots &\vdots&\vdots\\
0 & 0& 0 &\ldots&1\\
1 & 0& 0 &\ldots&0\\
\eb,
\end{align*}
it is well known that $P^m=I$ and therefore $n=m$.
We define the matrix operator $f$ as follows,
\begin{align*}
f(X)=P^\top X P,
\end{align*}
where $X\in\mathbb{R}^{m \times m}$.
After direct manipulation yield
\begin{align*}
\mathcal{A}=(AP^\top)^n ,\,\mathcal{B}=(PB)^n,\, \mathcal{C}=\sum\limits_{i=0}^{n-1}(AP^\top)^i C (PB)^i.
\end{align*}
Then, for the solution $\mathcal{X}$ of the Stein matrix equation~\eqref{Stein}
\begin{align}\label{ex}
\mathcal{X}=(AP^\top)^n \mathcal{X} (PB)^n+\sum\limits_{i=0}^{n-1}(AP^\top)^i C (PB)^i.
\end{align}
It follows that
\begin{align*}
F_{f}(\mathcal{X})=\dfrac{1}{n}\left( \sum\limits_{i=0}^{n-1} \left(AP^\top)^i ( \mathcal{X} +(n-i-1)C )(PB)^i\right) \right).
\end{align*}
According to the Theorem~\ref{solve eq},  $F_{f}(\mathcal{X})$ is a solution of the following linear matrix equation
\begin{align}\label{ex1}
X=Af(X)B+C=AP^\top X PB+C.
\end{align}
Under the condition of uniquely solvable for Eq.~\eqref{ex1}, that is, the set of $\sigma(AP^\top)$ and  the set of $\sigma(B^\top P^\top)$ are reciprocal free. This condition is coincident with the condition of uniquely solvable for Eq.~\eqref{ex}. By the direct calculation, it can be shown that $F_{f}(\mathcal{X})$ is exactly equal to $\mathcal{X}$.
\end{example}

\begin{example}\label{ex2}
Let $f$ be the (real) anti-transpose operator. It is meant to be reflect $X$ over its anti-diagonal (which runs from top-right to bottom-left) to obtain $f(X)$. That is, given a $m$-square matrix $X=[x_{ij}]\in\mathbb{R}^{m \times m}$, then
\begin{align*}
f(X)=[x_{m+1-j\,,\,m+1-i}].
\end{align*}
Note that the effect of applying the anti-transpose operator $f$ to $X$ is never equivalent to a combination of any row and column operators apply to $X$ . It is clear that $f$ is anti-linear and period-2. Moveover, the set of $\sigma(A)$ is equal to the set of $\sigma(f(A))$ by consider the Schur decomposition of $A$. From Corollary~\ref{Coro}, the necessary condition for the existence of unique solution of the linear matrix equation is as follows,
 \begin{center}
 $\sigma(B^\top f^\top(A))$ and $\sigma(Af(B))$ are reciprocal free,
 \end{center}
 or equivalently,
  \begin{center}
 $\sigma(Af(B))$ is reciprocal free.
 \end{center}
Now,  we consider the necessary and sufficient conditions for the unique solvability of the solution $X$.
With the Kronecker product, this linear equation can be written as
the enlarged linear system
\begin{align}\label{KronD}
(I_{m^2}-(B^\top\otimes A)\mathcal{P}) \mbox{vec}(X)=\mbox{vec}(C),
\end{align}
where $\mathcal{P}$ is the Kronecker-like permutation matrix~\cite{Bernstein2009} which maps
$\mbox{vec}(X)$ into $\mbox{vec}(f(X))$, i.e.,
\begin{equation*}
\mathcal{P}=
\sum\limits_{1\leq i,j\leq m}e_{m+1-j}e_{m+1-i}^\top \otimes e_ie_j^\top,
\end{equation*}
where $e_i$ denotes the $i$-th column of the $m\times m$ identity
matrix $I_{m}$. Due to the specific structure of $\mathcal{P}$ and then analogous to the consequences of \cite[Lemma~2.2]{Chiang2013AAA}, we can show that
$$\sigma((B^\top\otimes A)\mathcal{P})=
\left \{\lambda_i , \pm\sqrt{\lambda_i\lambda_j} |
 \lambda_i, \lambda_j \in \sigma(A f(B))  = \left\{ \lambda_1,\ldots,\lambda_m\right\},
1\leq i <  j\leq m \right\}.$$ Here,  $\sqrt{z}$ denotes the principal square root of a complex number $z$.
We thus have the following solvability conditions, the linear matrix equation Eq.~\eqref{eq} with the anti-transpose operator $f$ is
uniquely solvable if and only if the following conditions are
satisfied:
\begin{itemize}
\item[(1)] The set of $\sigma(A f(B)) \setminus \{-1\}$ is reciprocal free.
\item[(2)] $-1$ can be an eigenvalue of the matrix $A f(B)$, but must be simple.
\end{itemize}
\end{example}

\begin{example}\label{ex3}
In the final example, we consider a more general linear matrix equation
\begin{align}\label{extend}
X=\sum\limits_{i=0}^{N-1} A_i f_i (X) B_i+C,
\end{align}
where $A_i,B_i$ and $C$ are $m\times m$ complex matrices for $i=0,1,\cdots,N-1$, and $N$ is a given positive integer. The family of operators $\{f_i\}$ are linear maps from $\mathbb{C}^{m \times m}$ to itself with \emph{homogeneous period-$n$}. The set of operators $\{ f_0,f_1,\cdots,f_{N-1}\}$ is called homogeneous period-$n$ if the following two conditions are both satisfied.
\begin{itemize}
\item[a.] For any $i\neq j$, $f_i$ and $f_j$ are commuting. That is, $f_i \circ f_j=f_j \circ f_i$.
\item[b.] $ f_0^{(i_0)}\circ f_1^{(i_1)}\circ\cdots\circ f_{N-1}^{(i_{N-1})}=I$ for arbitrary nonnegative integers $i_0,\cdots,i_{N-1}$ such that $i_0+\cdots+i_{N-1}=n$.
\end{itemize}
 We write the notation of the composition of operators $f_{k_0},\cdots,f_{k_{i-1}}$  by $f_{k_0,\cdots,k_{i-1}}$ and  introduce the important quantities $K_i(X)$ as follows,
\begin{align*}
 &K_i(X)=\sum\limits_{0\leq k_0,\ldots,k_{i-1} \leq N-1} \prod_{j=1}^i f_{k_0,\cdots,k_{j-2}}(A_{k_{j-1}}) \left[f_{k_0,\cdots,k_{j-1}}(X)\right] \prod_{j=1}^i f_{k_{0},\cdots,k_{i-j-1}}(B_{k_{i-j}})\\
\end{align*}
for any positive integer $1\leq i \leq n$ and arbitrary $m$-square matrix $X$ and we define $K_0(X)=X$. 
Note that $K_{i+j}(X)=K_i(K_j(X))=K_j(K_i(X))$ like the behavior of the operator $G_i$.
 After a composition of $n$ copies of the right hand side of Eq.~\eqref{extend}, we can be transformed Eq.~\eqref{extend} into the generalized Stein matrix equation,
\begin{align}\label{exstein}
\mathcal{X}=K_n(\mathcal{X})+\sum\limits_{i=0}^{n-1} K_i(C).
\end{align}
Similarly, we let $\mathcal{X}$ be the solution of Eq.~\eqref{exstein} and
\begin{align*}
F_{f}(\mathcal{X})=\dfrac{1}{n}\left( \sum\limits_{i=0}^{n-1} \left(K_i(\mathcal{X})+(n-i-1)K_i(C)\right) \right)
\end{align*}
associate with Eq.~\eqref{ex}. It is analogous to show that $F_{f}(\mathcal{X})$ is also a solution of Eq.~\eqref{extend} and thus also solve Eq.~\eqref{exstein}. A special
case is that the family of operators $\{f_i\}$ are the same $f\equiv f_{0}=\cdots=f_{N-1}$. Then,
\begin{align*}
 K_i(\mathcal{X})&=\sum\limits_{0\leq k_0,\ldots, k_{i-1} \leq N-1}  \prod_{j=1}^i f^{(j-1)}(A_{k_{j-1}}) \left[f^{(j)}(\mathcal{X})\right] \prod_{j=1}^i f^{(i-j)}(B_{k_{i-j}}),
\end{align*}
which is exactly equal to the matrix $G_i(\mathcal{X})$ when $N=1$.
\end{example}

%
%
%
%

\section{Concluding remark}
In this paper, we are mainly interested how the solution of Eq.~\eqref{eq} can be found by computed the solution of a Stein matrix equation and characterize this transformation. We suggest an approach for solving the general solution of Eq.~\eqref{eq}. A formulation of a class of the linear matrix equation is already available in the recent work \cite{Zhou2011}. The new framework is presented in this paper to calculate the solution of Eq.~\eqref{eq} with a  general class of linear matrix operator $f$. Two expressions \eqref{prove1} and \eqref{prove2} of this formulation corresponding to other equivalent linear matrix equations are supplied. At the end, we provide two special matrix operator $f$ to perform our theory and the similar result of a generalization linear matrix equation is also mentioned.

For future work, some interesting problems are treated. The first one is that how the relationship between the conditions of uniquely solvable of Eq.~\eqref{eq} and the conditions of uniquely solvable of Eq.~\eqref{Stein}?
The recent examples in the literature~\cite{Zhou2011,Chiang2013AAA} and Example~\ref{ex2} tell us that they are almost the same. What properties of the matrix operator $f$ such that the conditions of uniquely solvable of Eq.~\eqref{eq} and the conditions of uniquely solvable of Eq.~\eqref{Stein}
 \begin{center}
 $\sigma(\mathcal{A})$ and $\sigma(\mathcal{B})$ are reciprocal free,
 \end{center}
are matched and vice versa? It seem possible that two conditions are the same arising from the relationship between the shape of $\mathcal{S}_1$ and the shape of $\mathcal{S}_2$. This will be further explored in the future. The other challenge problem is that we try to solve Example~\ref{ex3} without the assumption of homogenous periodic operators. The assumption of the open problem has only required the periodically of each operator $f_i$. All these questions are under investigation and will be reported elsewhere.
\section*{Acknowledgement}
This research work is partially supported by the National Science Council and
the National Center for Theoretical Sciences in Taiwan.

%


\begin{thebibliography}{1}
\expandafter\ifx\csname url\endcsname\relax
  \def\url#1{\texttt{#1}}\fi
\expandafter\ifx\csname urlprefix\endcsname\relax\def\urlprefix{URL }\fi
\expandafter\ifx\csname href\endcsname\relax
  \def\href#1#2{#2} \def\path#1{#1}\fi

\bibitem{Zhou2011}
B.~Zhou, J.~Lam, G.-R. Duan,
  \href{http://dx.doi.org/10.1016/j.laa.2011.03.003}{Toward solution of matrix
  equation {$X=Af(X)B+C$}}, Linear Algebra Appl. 435~(6) (2011) 1370--1398.
\newblock \href {http://dx.doi.org/10.1016/j.laa.2011.03.003}
  {\path{doi:10.1016/j.laa.2011.03.003}}.
\newline\urlprefix\url{http://dx.doi.org/10.1016/j.laa.2011.03.003}

\bibitem{Chiang2012}
C.-Y. Chiang, E.~K.-W. Chu, W.-W. Lin,
  \href{http://dx.doi.org/10.1016/j.amc.2012.01.065}{On the
  {$\star$}-{S}ylvester equation {$AX\pm X^\star B^\star=C$}}, Appl. Math.
  Comput. 218~(17) (2012) 8393--8407.
\newblock \href {http://dx.doi.org/10.1016/j.amc.2012.01.065}
  {\path{doi:10.1016/j.amc.2012.01.065}}.
\newline\urlprefix\url{http://dx.doi.org/10.1016/j.amc.2012.01.065}

\bibitem{Zhou2009}
B.~Zhou, J.~Lam, G.-R. Duan,
  \href{http://dx.doi.org/10.1016/j.aml.2009.01.012}{On {S}mith-type iterative
  algorithms for the {S}tein matrix equation}, Appl. Math. Lett. 22~(7) (2009)
  1038--1044.
\newblock \href {http://dx.doi.org/10.1016/j.aml.2009.01.012}
  {\path{doi:10.1016/j.aml.2009.01.012}}.
\newline\urlprefix\url{http://dx.doi.org/10.1016/j.aml.2009.01.012}

\bibitem{Bernstein2009}
D.~S. Bernstein, Matrix {M}athematics, {T}heory, {F}acts, and {F}ormulas, 2nd
  Edition, Princeton University Press, Princeton, NJ, 2009.

\bibitem{Chiang2013AAA}
C.-Y. Chiang, \href{http://www.hindawi.com/journals/aaa/2013/824641/}{A note on
  the $\top$-stein matrix equation}, ABSTR. APPL. ANAL. 2013~(Article ID
  824641).
\newline\urlprefix\url{http://www.hindawi.com/journals/aaa/2013/824641/}

\end{thebibliography}
\end{document}